\documentclass[11pt]{birkjour}

\usepackage{CormorantGaramond}

\usepackage{mathtools}
\usepackage[table,xcdraw]{xcolor}
\usepackage{tikz}
\usepackage{tikz-cd}
\usepackage{mathdots}
\usepackage{multicol}
\usepackage{color}
\usepackage{faktor}
\usepackage{textcomp}
\usepackage{tkz-graph}
\usepackage{tikz-cd}
\usetikzlibrary{calc}
\usepackage{tikz}
\usetikzlibrary{arrows}
\usepackage{stackrel}
\usepackage{faktor}
\usepackage{stackengine}
\usepackage{amssymb}
\usepackage{hyperref}
\usepackage[mathscr]{euscript}
\stackMath 

\DeclareMathOperator{\diag}{diag}
\DeclareMathOperator{\SL}{SL}
\DeclareMathOperator{\GL}{GL}
\DeclareMathOperator{\Lie}{Lie}
\DeclareMathOperator{\ad}{ad}

\DeclareMathOperator{\Char}{char}

\DeclareMathOperator{\Tr}{Tr}
\DeclareMathOperator{\End}{End}
\DeclareMathOperator{\Span}{span}
\DeclareMathOperator{\M}{Mat}

\newcommand{\C}{\mathbb{C}}
\newcommand{\Z}{\mathbb{Z}}

\newcommand{\A}{\alpha}
\newcommand{\B}{\beta}

\newcommand{\D}{\delta}
\newcommand{\E}{\varepsilon}
\newcommand{\p}{\rho}
\newcommand{\La}{\lambda}

\newcommand{\Si}{\sigma}

\newcommand{\frakg}{\ensuremath{\mathfrak{g}}}
\newcommand{\g}{\ensuremath{\mathfrak{g}}}
\newcommand{\gl}{\ensuremath{\mathfrak{gl}}}
\newcommand{\fraksl}{\ensuremath{\mathfrak{sl}}}
\newcommand{\frakgl}{\ensuremath{\mathfrak{gl}}}

\newcommand{\fh}{\ensuremath{\mathfrak{h}}}

\newcommand{\lra}{\ensuremath{\longrightarrow}}

 \newtheorem{thm}{Theorem}[section]
 \newtheorem{cor}[thm]{Corollary}
 
 \newtheorem{prop}[thm]{Proposition}
 \theoremstyle{definition}
 
 \theoremstyle{remark}
 \newtheorem{rem}[thm]{Remark}
 
 \numberwithin{equation}{section}

\begin{document}

%
%
%
%
%
%
%
%
%
\title[Standard automorphisms of $\fraksl(n+1,\C)$ and their relations]
{Standard automorphisms of $\mathfrak{sl}(n+1,\mathbb{C})$ and their relations}
\author[DRM]{David Reynoso-Mercado}
\address{Instituto de Matem\'aticas, Facultad de Ciencias Exactas y Naturales, Universidad de Antioquia UdeA,\br Calle 70 No. 52-21, Medell\'in, Colombia}
\email{david.reynoso@udea.edu.co}
\thanks{This work was completed with the support of CONACYT}

\date{February 28, 2023}
\begin{abstract}
Let $ \g $ be the simple Lie algebra of square matrices $ (n+1) \times (n+1) $ with zero trace. There are certain relations concerning standard automorphisms that are considered ``folklore". One can find proof of these in Millson and Toledano Laredo's work [Transformation Groups, Vol. 10, No. 2, 2005, pp. 217-- 254], but the mentioned proof is based on topological arguments which are non-trivial. The aim of this work is to verify these relations in an elementary way.
\end{abstract}
\maketitle
\section{Introduction}

Let $\g\coloneqq\fraksl(n+1,\C)$ be the simple Lie algebra of square matrices $(n+1)\times (n+1)$ with zero trace. We denote the standard basis of the space of square matrices $\M_{n+1}(\C)$ by $\{E_{i,\,j}\}_{i,j=1}^{n+1}$ and the set $\{1,\,2,\cdots,\,n\}$ by $J_n$.

As a Lie algebra, $\g$ is generated by the elements $e_i\coloneqq E_{i,\,i+1}$ and $f_i\coloneqq E_{i+1,\,i}$ for $i\in J_n$.

Note that with the corresponding simple co-roots $h_i\coloneqq E_{i,\,i}-E_{i+1,\,i+1}$, with $i\in J_n$, we have that $[e_i,f_i]= h_i$ and $\fh\coloneqq \bigoplus_{i=1}^n\C h_i$ is a Cartan subalgebra of $\g$. 

Let $D\coloneqq \bigoplus_{i=1}^{n+1} \C E_{i,i}$ and let $\fh$ be a vector subspace of $D$. The dual base of the standard base of $D$ is denoted as $\{\tilde{\E}_i\}_{i\in J_n}$, and the corresponding elements of $\fh^*$ are denoted as $\E_i=\tilde{\E}_i|_{\fh}$.

With this notation $\Phi\coloneqq\{\E_i-\E_j|1\leq i, j\leq n+1\}\setminus\{0\}\subset \fh^*$
is the root system of $\g$, that is, the $\C E_{i,\,j}=\{x\in\g|[h,x]=(\E_i-\E_j)(h)\cdot x\,\,\,\,\forall h \in\fh\}$
are the root spaces of $\g$ with respect to $\fh$. The Weyl group of $\g$, denoted by $\mathcal{W}$, is identified with the symmetric group $S_{n+1}$, which acts linearly on $\fh^*$ such that for $\Si\in \mathcal{W}$
$$\Si(\E_i-\E_j)=\E_{\Si(i)}-\E_{\Si(j)}.$$

The Weyl group $\mathcal{W}$ is generated by the simple transpositions $s_i=(i,i+1)$ with $i\in J_n$, and the following relations form a complete system of relations 
\begin{eqnarray}\label{relsim}
(s_is_j)^{m_{ij}}=1\mbox{ with } i,\,j\in J_n \mbox{ and}
\end{eqnarray} 
\[m_{ij}=\left\{\begin{array}{ll}1 & \mbox{if } i-j=0,\\3 & \mbox{if } |i-j|=1,\\ 2  &\mbox{otherwise}.\end{array}\right.\]

Let us denote the adjoint representation of $\g$ by $\ad\colon\g\lra\frakgl(\g)$, $x\mapsto \ad_x\coloneqq [x,-]$, and define\begin{eqnarray}\label{tau}
 \tau_i\coloneqq \exp(\ad_{e_i})\exp(\ad_{(-f_i)})\exp(\ad_{e_i})\,\,\,\forall i\in J_n.
 \end{eqnarray}
For all $i \in J_n$, it is easy to see that $\tau_i(\C E_{k,\,l})=\C E_{s_i(k),\,s_i(l)}$ for all $k,\, l\in J_n$ with  $k\neq l$. In \cite[21.2]{Hum},  we can see that the weight spaces of the adjoint representation are permuted by the $\tau_i$ in the same way that the corresponding weights (roots in this case) are permuted by the $s_i$, that is $\tau_i(\fh)=\fh$.

The aim of this work is to verify the following realations by the $\tau_i$'s in an elementary way.
\begin{thm}\label{theorem1} Let $n$ be a positive integer and let $\tau_i$'s be the automorphisms  of $\fraksl(n+1,\C)$ defined in  (\ref{tau}). Then the automorphisms $\tau_i$, $1\leq i\leq n$, fulfil the relations
\begin{eqnarray}
\underbrace{\tau_i\tau_j\cdots}_{m_{ij}}&=&\underbrace{\tau_j\tau_i\cdots}_{m_{ij}},\label{0.2}\\
\tau_i^2\tau_j^2&=&\tau_j^2\tau_i^2,\label{0.4}\\
\tau_i^4&=&1,\\
\label{0.6}\tau_i\tau_j^2\tau_i^{-1}&=&\tau_j^2\tau_i^{-2\A_j(h_i)},
\end{eqnarray}
for any $1\leq i \neq j\leq n$, where $m_{ij}$ is the same for the relation in (\ref{relsim}).
\end{thm}
Note that the relations (\ref{relsim}) for $i=j$, we have imply $s_i^2=1$. Under this premise, the remaining relations in (\ref{relsim}) are equivalent to (\ref{0.2}), also known as the braid relations.

It is evident that directly verifying the relations of Theorem \ref{theorem1} is a challenging task. To prove our result, we follow the steps of \cite[Proposition 2.14]{Tol}, which is a more general result.

We will study the elements
$$\Si_i\coloneqq \exp({e_i})\exp(-f_i)\exp({e_i})\in \SL(n+1,\C)$$
and verify that these elements of the simple connected Lie group $\SL(n+1,\C)$ fulfill ``the same'' relations as the elements $\tau_i$ of Theorem \ref{theorem1} inspired by the results of Millson and Toledano Laredo \cite{Tol}. It is important to note that the ``braid relations'' in this context are classical (\cite{chevalley},\cite{Br}), while the other relations are considered ``folklore''.

The theorem's result will follow once we observe that
$$\tau_i(x)=\Si_i x \Si_i^{-1}\,\,\, \forall i\in J_n, x\in\fraksl(n+1,\C).$$

It is important to note that the automorphisms $\tau_i$ can be defined not only for the adjoint representation, but also for any finite-dimensional representation $V$ of $\g$, see \cite[21.2]{Hum}. In this context, they are relevant when one wants to see that the character of $V$ is invariant under the action of the Weyl group.

\section{Preliminaries}\label{1sec}
\subsection{Braid group}
The {\em Artin braid group} $B_n$ is the group generated by the $n-1$ elements $\{\Si_i\}_{i=1}^{n-1}$ and the braid relations:
\begin{enumerate}
\item $\Si_i\Si_j=\Si_j\Si_i$, for all $i$ and $j$ in $\{1,\,2,\cdots,\, n-1\}$, with $|i-j|\geq 2$, and
\item $\Si_i\Si_{i+1}\Si_i=\Si_{i+1}\Si_i\Si_{i+1}$, for  $i$ in $\{1,\,2,\cdots,\, n-2\}$.
\end{enumerate}

 Let $S_n$ be the symmetric group and $s_1,\,s_2,\cdots,\,s_{n-1}\in S_n$ its generators, where  $s_i$ is  the adjacent transposition that swaps  $i$ and $i+1$. It is known that there is a unique $\pi:B_n\lra S_n$ group homomorphism such that $\pi(\Si_i)=s_i$,  see \cite[Lemma 1.2]{KT}, the $\pi$ is called the {\em natural projection}. Moreover, since $\{s_i\}_{i=1}^{n-1}$ is the generating set of $S_n$, we have that $\pi$ is surjective. 
  
The kernel of $\pi:B_n\lra S_n$ is called the {\em pure braid group} and is denoted by $P_n$. The elements of $P_n$ are called {\em pure braids on $n$ strings}.
\subsection{Lie algebras and their representations}\label{3}
 
Let $X$ be a complex matrix. Recall that the exponential of $X$, denoted by $\exp(X)$  or $e^X$, is  the usual power series
$$\exp(X)=\sum_{m=0}^{\infty} \; \frac{X^m}{m!}.$$
\begin{prop} For all $X\in \M_n(\C)$, the series $\exp(X)$  converges and $\exp$ is a continuous function of $X$.
\end{prop}
\begin{proof}See \cite[Section 2.1]{B.C.H}.\end{proof}

Let $V$ be a finite dimensional vector space over a field $k$ and let $\End V$ denote the set of linear transformations $V\rightarrow V$. We write  $\mathfrak{gl}(V)$ for $(\End_k(V))_L$, viewed as Lie algebra and call it the \emph{general linear algebra}.  If $V=k^n$, then $\End_k(k^n)\cong\M_n(k)$, and we use $\mathfrak{gl}(n,k)$ to denote $\M_n(k)$. Recall that the \emph{special linear Lie algebra} denoted by $\mathfrak{sl}(n,k)$ is the set of matrices $X\in\mathfrak{gl}(n,k)$ such that $ \Tr(X)=0$, where $\Tr(X)\coloneqq \sum_{i=1}^nx_{i,\,i}$ is the trace of $X$.

A \emph{representation} of a Lie algebra $\frakg$ is a pair $(V,\p)$, where $V$ is a  vector space over $k$  and $\p:\frakg\lra\frakgl(V)$ is a homeomorphism of Lie algebras, that is $\p$ is a linear transformation such that $\p([x,y])=[\p(x),\p(y)]$. If $\g$ is a Lie algebra and $x\in \g$, we define the \emph{adjoint representation} of $\frakg$ as $\ad:\frakg\lra\frakgl(\frakg)$ where  $x\mapsto \ad_x\coloneqq [x, -]$ is a homomorphism of Lie algebras.

\begin{rem}\label{remark1}Let $X,\,Y \in\M_n(\C)$. Then $\exp(\ad_{X})(Y)=\exp(X)\,Y\exp(-X)$. 
\end{rem}

Let $e,\,f,\,h$ be matrices of $\fraksl(2,k)$, defined as follows
\begin{eqnarray}\label{efh}
    e\coloneqq  \left(\begin{array}{cc} 0& 1 \\ 0 & 0 \\  \end{array} \right),\;h\coloneqq  \left(\begin{array}{cc} 1& 0 \\ 0 & -1 \\  \end{array}\right),\;f\coloneqq  \left(\begin{array}{cc} 0& 0 \\ 1 & 0 \\  \end{array} \right).
\end{eqnarray}
Then  the set $\{e,\,h,\,f\}$ is a basis of $\fraksl(2,k)$, such that  $[h,e]=2e$ $[h,f]=-2f$, $[e,f]=h$. The following theorem is taken from \cite[Theorem 1.2.14]{Soergel}.

\begin{thm}[Simple representations of $\fraksl(2,k)$] Let $k$ be a field of characteristic zero. Then:
    \begin{enumerate}
        \item For every positive finite dimension, there exists a unique (up to isomorphism) simple representation of the Lie algebra $\mathfrak{sl}(2, k)$ of that dimension.
        
        \item Let $e, h, f$ be a basis of $\mathfrak{sl}(2, k)$ satisfying the commutation relations:
        \[
        [h, e] = 2e, \quad [h, f] = -2f, \quad [e, f] = h.
        \]
        Every simple representation $L$ of dimension $m + 1$ decomposes under $h$ into one-dimensional eigenspaces:
        \[
        L = L_m \oplus L_{m-2} \oplus \cdots \oplus L_{-m+2} \oplus L_{-m},
        \]
        corresponding to the integer eigenvalues $m, m-2, \ldots, -m+2, -m$. Furthermore, the following hold: If $L_j \neq 0 \neq L_{j+2}$, then $f : L_j \to L_{j-2}$ and $e : L_j \to L_{j+2}$.
    \end{enumerate}
\end{thm}

\subsection{Root space decomposition}\label{4} 

Let $\g\coloneqq\gl(n,k)$ and let $\fh\subset\g$ be the subset of diagonal matrices. Then, $\fh$ is subalgebra of $\g$. Let $\p\coloneqq\ad_\g |_\fh:\fh\lra\gl(\g)$ be a representation of $\fh$ with $\p(x)$ semisimple for all  $x\in\fh$. Let  $x=\diag(x_1,x_2,\cdots,x_n)\in\fh$ and $\{E_{i,\,j}\}_{i,j=1}^n$ be the standard basis of $\M_{n}(\C)$, where $E_{i,\,j}$ be the matrix whose $ij$-th entry is $1$ and all other entries are $0$. In this case, we have the commutation relation $[x,E_{i,\,j}]=(x_i-x_j)E_{i,\,j}$. We can define a map $\E_i:\fh\lra\C$ by  setting $\E_i(x)= x_i$. Thus we obtain:
$$[x,E_{i,\,j}]=(\E_i-\E_j)(x) E_{i,\,j}.$$
Let $P(\g)$ be the set of weights of $\g$. It is easy to see that 
$$P(\g)=\{\E_i-\E_j | 1\leqslant i\leqslant j \leqslant n\}.$$
For $i\neq j$, we have $\g_{\E_i-\E_j}= \C E_{i,\,j}$, and $\g_0=\fh$, provided that $\Char(k)\neq 2$.

 Let $\g$ a semisimple Lie algebra over $\C$ and, let $\fh\subset\g$ be a subalgebra. We call  $\fh$ a \emph{Cartan subalgebra} if and only if $\fh$ is abelian, consists of semisimple elements, and is maximal with this property. Let $\fh$ be a subalgebra of $\g$ a semisimple Lie algebra over $\C$. We will write $\langle\La,h\rangle$ to denote  $\La(h)$ for all $\La\in \fh^*$ and  $h\in\fh$, and we use $\p$ to denote  the adjoint map $\ad_\g |_\fh:\fh\lra\gl(\g)$. Since $\fh$ is Cartan subalgebra, its elements are semisimple. Therefore, by \cite[Proposición 2.2.2]{Soergel}, we know that $\p(h)$ is semisimple $\forall h\in\fh$. Thus,  $\g$ decomposes as $\bigoplus_{\La\in\fh^*} \g_\La$, where $\g_\La=\{x\in\g | [h, x]= \langle\La,h\rangle (x)\,\,\forall h\in\fh\}.$

Now, consider $\g=\fraksl(n+1,\C)$ and  $\fh\subset\g$ the subalgebra consisting of diagonal matrices with trace zero. Then, $\A_i=\E_i-\E_{i+1}$ with $i\in\{1,\cdots,\, n\}$  forms a basis of $\fh^*$. Moreover, we have that $\Phi(\g,\fh)=\{\A_{ij}=\E_i-\E_j |i\neq j\}$. Thus, 
$$\fraksl(n+1,\C)=\fh\oplus\bigoplus_{\A_{ij}\in\Phi} \C E_{i,\,j}.$$
The following theorem is taken from  \cite[Theorem 2.3.12]{Soergel}.
\begin{thm}\label{theorem5.1}Let $\g$ be a semisimple Lie algebra over $\C$, $\fh\subset\g$ a Cartan subalgebra and  $\Phi$  the corresponding root system. The following statements hold:
\begin{enumerate}
\item $\fh$ is itself centralizer. 
\item All root spaces have one dimension. Moreover, for all $\A\in\Phi$,  there exists an injective Lie algebras homomorphism $\varphi_\A:\fraksl(2,\C)\lra\g$ defined as
$$  \varphi_\A\left(\left( \begin{array}{cc} 0 & \C \\ 0 & 0 \end{array} \right)\right)=\g_\A\mbox{, } \varphi_\A\left(\left( \begin{array}{cc} 0 & 0 \\ \C & 0 \end{array} \right)\right)=\g_{-\A}\mbox{ y }\varphi_\A\left(\C\left( \begin{array}{cc} 1 & 0 \\ 0 & -1 \end{array} \right)\right)=[\g_\A,\g_{-\A}]\mbox{.}$$
In particular, $\A\in\Phi$ if and only if $-\A\in\Phi$.
\item If $\A \in\Phi$, then $\C\A\cap\Phi=\{\A,-\A\}$.
\item If  $\A\mbox{, }\B\in\Phi\mbox{ and }\A+\B\in\Phi,$ then $[\g_\A,\g_\B]=\g_{\A+\B}$.
\end{enumerate}
\end{thm} 
 For $\A\in\Phi$, we can define a {\em co-root} $h_\alpha$ as an element in $\fh$ such that $h_\A\in[\g_\A,\g_{-\A}]$ and $\langle\A,h_\A\rangle=2$.
\begin{thm}\label{theorem5.2} Let $\g$, $\fh$ and $\Phi$ be as in Theorem \ref{theorem5.1}. The following statements hold: 
\begin{enumerate}
\item For  any $\A\mbox{, }\B\in\Phi$, we have that $\langle\B,h_\A\rangle\in\Z$ and $s_\A(\B)\coloneqq\B-\langle\B,h_\A\rangle\A\in\Phi$.
\item $\Span_\C\Phi=\fh^*$.
\end{enumerate}
\end{thm}
\begin{proof}See \cite[Theorem 2.3.17]{Soergel}.\end{proof}
For the Lie algebra $\fraksl(n+1,\C)$ we have that $\Phi=\{\A_{ij}=\E_i-\E_j |i\neq j\}$ its co\-rresponding root system and $\Pi=\{\A_{i}=\E_i-\E_{i+1} |1\leq i\leq n\}$ is the set of simple roots.

Let $\mathcal{W}$ be the Weyl group of $\Phi$. We want to show that $\mathcal{W}=\langle\{s_{\A_i}\mid1\leq i\leq n\}\rangle$. Due to the definition of Weyl group, it is enough to show that $s_{\A_{ij}}\in\langle\{s_i\mid1\leq i\leq n\}\rangle$ for all $\A_{ij}\in \Phi$. First, from the definition of $s_\A(\B)$ (see theorem \ref{theorem5.2}), we have
\begin{eqnarray}\label{eq5.1}
s_{\A_i}(\A_j)= \left\{ \begin{array}{cl}
-\A_i&\mbox{ if } j=i,\\  
\A_{ij}&\mbox{ if } j=i+1,\\ 
\A_{ji}&\mbox{ if } j=i-1,\\ 
\A_j&\mbox{ if } |j-i|\geq 2,
\end{array}
\right. 
\end{eqnarray}
for any $1\leq i,\,j\leq n$. On other hand, for any $1\leq i< j\leq n$ and $1\leq k\leq n$, we have
\begin{eqnarray}\label{eq5.2}
s_{\A_{ij}}(\A_k)= \left\{ \begin{array}{cl}
\A_{ik}&\mbox{if } k=j+1,\\  
\A_{kj}&\mbox{if } k=i-1,\\ 
-\A_{ij-1}&\mbox{if }k=j,\\ 
-\A_{i+1j}&\mbox{if }k=i,\\ 
\A_k&\mbox{otherwise.}
\end{array}
\right. 
\end{eqnarray}

Using equation (\ref{eq5.1}) and (\ref{eq5.2}), it is easy to verify that
$$s_{ij}=s_is_{i+1}\cdots s_{j-i}s_js_{j-1}\cdots s_{i+1}s_i\in\langle\{s_i\mid1\leq i\leq n\}\rangle.$$
Therefore, $\mathcal{W}=\langle\{s_i\mid1\leq i\leq n\}\rangle$. Moreover, simple computations reveal that $\mathcal{W}$ is isomorphic to the symmetric group $S_{n+1}$.
\section{Tits extension}\label{7}

In this section, we will recall the arguments about generalised braid groups presented by Millson and Toledano (\cite[Section 2]{Tol}). 

Let $\g$ be a complex, semisimple Lie algebra with Cartan subalgebra $\fh$ and $\mathcal{W}$ be the Weyl group. Let $\Phi$ be a root system, $\Pi=\{\A_1,\cdots,\A_n\}$ a basis of $\Phi$ and $V$ be a finite-dimensional $\g$-module. We denote by $\fh_{reg}\coloneqq\fh\setminus\bigcup_{\A\in \Phi} \ker(\A)$ the set of regular elements. 

Fix a basepoint $t_0\in\fh_{reg} $. Let  $[t_0]$ be its image in $\fh_{reg}/ \mathcal{W}$, and let $P_\g=\pi_1(\fh_{reg}; t_0)$, $B_\g=\pi_1(\fh_{reg}/ \mathcal{W}; [t_0])$ be the generalized pure and full braid groups of type $\g$. The fibration $\fh_{reg}\lra\fh_{reg}/\mathcal{W}$ gives rise to the exact sequence
 $$1\lra P_\g\lra B_\g\lra \mathcal{W}\lra 1,$$
with $B_\g\lra \mathcal{W}$ is obtained by associating to $p\in B_\g$ to the unique $w\in \mathcal{W}$ such that $w^{-1}t_0=\tilde{p}(1)$, where $\tilde{p}$ is the unique lift of $p$ to a path in $\fh_{reg}$ such that $\tilde{p}(0)=t_0$.

Let  $G$ be the complex, connected, and simply-connected Lie group with Lie algebra $\g$, $T$  its torus with Lie algebra $\fh$, and $N(T)\subset G$  the normalizer of $T$ in $G$, so that $\mathcal{W}\cong N(T)/T$. We regard $B_\g$ as acting on $V$ by choosing a homeomorphism $\Si:B_\g\lra N(T)$ compatible with the diagram in Figure \ref{diagramas}. 
\begin{figure}[ht]
\centering
\begin{tikzpicture}
[->,>=stealth',shorten >=1pt,auto,node distance=2cm,thick,main node/.style=]
  \node[main node] (1) {$B_\g$};   
  \node[main node] (2) [right of=1] {$N(T)$};
  \node[main node] (3) [below of =2]{$\mathcal{W}$};   
\path[every node/.style={font=\sffamily\small}]     
(1) edge node   {$\Si$} (2)         
     edge node         {} (3)           
(2) edge node  {} (3); 
\end{tikzpicture}
\caption{Diagram} 
\label{diagramas}
\end{figure}

By Brieskorn's theorem (see \cite[page 57]{Br}), $B_\g$ is presented on generators $S_1,S_2,\cdots,S_n$, labelled by the simple reflections $s_1,\cdots, s_n\in \mathcal{W}$, with relations
\begin{equation}\label{trenzas}\underbrace{S_iS_j\cdots}_{m_{ij}}=\underbrace{S_jS_i\cdots}_{m_{ij}},\end{equation}
where $m_{ij}$ is the order of $s_is_j$ in the Weyl group. One of the first works in which the relations from (\ref{trenzas}) appears is \cite[Lemma 56]{chevalley}.

Tits has given a simple construction of a canonical class of homomorphisms $\Si$ that differ from each other via conjugation by an element of $T$, see \cite{Ti}. 
\section{Proof of Theorem \ref{theorem1}}\label{6}
This section will divide the proposition of Milson and Toledano \cite[Proposition 2.14]{Tol}, into three new propositions to verify the relations (\ref{0.2})-(\ref{0.6}) in our particular case, the Lie algebra $\fraksl(n+1, \C)$.

From this point on, we will be working with the Lie group of matrices $G=\SL(n+1;\C)$, with $\g=\fraksl(n+1, \C)$ its Lie algebra. For $\A_i$ a simple root as in section \ref{4}, we can define          
$$G_i =\left\{
                                  \left( \begin{array}{c|c|c}
                                  I_{i-1} &0_{i-1\times 2} &0_{i-1\times n-i}   \\\hline
                                  0_{2\times i-1}&A&0_{2\times n-i}\\\hline
                                0_{n-i\times i-1} &0_{n-i\times 2}  & I_{n-i} 
                                  \end{array} \right) | \,A\in\SL(2,\C)\right\}\subset G,$$
where $I_j$ is the identity matrix in $\GL(j;\C)$ and the matrix $0_{kl}\in\M_{k\times l}(\C)$ is composed entirely of zeros.

It is easy to see that $G_i\cong \SL(2;\C)$ and that the Lie algebra $\Lie(G_i)$ is generated by
$$e_i=
                                 \left( \begin{array}{c|c|c}
                                  I_{i-1} &0_{i-1\times 2} &0_{i-1\times n-i}   \\\hline
                                  0_{2\times i-1}&e&0_{2\times n-i}\\\hline
                                0_{n-i\times i-1} &0_{n-i\times 2}  & I_{n-i} 
                                  \end{array} \right)\mbox{,  }h_i=
                                  \left( \begin{array}{c|c|c}
                                  I_{i-1} &0_{i-1\times 2} &0_{i-1\times n-i}   \\\hline
                                  0_{2\times i-1}&h&0_{2\times n-i}\\\hline
                                0_{n-i\times i-1} &0_{n-i\times 2}  & I_{n-i} 
                                  \end{array} \right)$$  and  $$f_i =
                                  \left( \begin{array}{c|c|c}
                                  I_{i-1} &0_{i-1\times 2} &0_{i-1\times n-i}   \\\hline
                                  0_{2\times i-1}&f&0_{2\times n-i}\\\hline
                                0_{n-i\times i-1} &0_{n-i\times 2}  & I_{n-i} 
                                  \end{array} \right) ,$$                                  
where $e,\,f,\,h\in\M_2(\C)$ are as in Equation (\ref{efh}).

For any $i\in\{1,\cdots,n\}$, let $T_i\coloneqq\{\exp(xh_i) |\,x\in\C\}\subset G_i$, for $1\leq i\leq n$, be an algebraic torus or more precisely:

$$T_i=\left\{   \left( \begin{array}{c|c|c}
                                  I_{i-1} &0_{i-1\times 2} &0_{i-1\times n-i}   \\\hline
                                  0_{2\times i-1}&\exp(xh)&0_{2\times n-i}\\\hline
                                0_{n-i\times i-1} &0_{n-i\times 2}  & I_{n-i} 
                                  \end{array} \right) | x\in\C\right\}.$$
Let $T\subset \SL(n+1;\C)$ be the set of diagonal matrices with determinant equal to $1$. 
\begin{prop}
    Let $T_i$ be the torus of $G_i$ group as above, then $T$ is equal to $\langle\bigcup_{i=1}^n T_i\rangle$.
\end{prop}
\begin{proof}
Since it is clear that any set $T_i$  is a subset of $T$, then $\langle\bigcup_{i=1}^n T_i\rangle\subset T$. 

On other hand, for $X=\diag(x_1,\cdots,x_{n+1})$ an element of $T$, for any $i\in\{1,\cdots,n\}$ there exists  a complex number $y_i$ such that $e^{y_i}=x_i$. We can define 
$$Y_i=\left( \begin{array}{c|c|c}
                                  I_{i-1} &0_{i-1\times 2} &0_{i-1\times n-i}   \\\hline
                                  0_{2\times i-1}&\exp(\sum_{j=1}^iy_jh)&0_{2\times n-i}\\\hline
                                0_{n-i\times i-1} &0_{n-i\times 2}  & I_{n-i} 
                                  \end{array} \right)\in T_i,$$
and since the determinant of $X$ is equal to $1$, we have $x_{n+1}=\Pi_{i=1}^nx_i^{-1}=e^{-\sum_{j=1}^ny_j}$. Then, 
$$ \Pi_{i=1}^nY_i=\diag(e^{y_1},e^{y_2},\cdots, e^{-\sum_{j=1}^ny_j})=\diag(x_1,x_2,\cdots,x_{n+1})=X.$$
Thus, $X\in\langle\bigcup_{i=1}^n T_i\rangle$ and $\langle\bigcup_{i=1}^n T_i\rangle=T$. 
\end{proof}
Obviously, the Lie algebra of $T$  is equal to $\fh\subset\g$, where $\fh$ is the subgroup of diagonal matrices which trace equal to zero.  

\begin{prop}
Let $N(T)$ be the normalizer of $T$ in $\SL(n+1;\C)$. Then, we have the following equality
\begin{eqnarray}\label{normalizer}
N(T)=\left\{\sum_{i=1}^{n+1}x_iE_{\sigma(i)}|\sigma\in S_{n+1}\mbox{ and }(\Pi_{i=1}^{n+1}x_i)\det\left(\sum_{i=1}^{n+1}E_{\sigma(i)}\right)=1\right\},    
\end{eqnarray}
where $E_{\Si(i)}\coloneqq E_{\Si(i)i}\in\M_{n+1}(\C)$. 
\end{prop}
\begin{proof}
    Let's take $X=\sum_{i=1}^{n+1}x_iE_{\sigma(i)}$ in the right set of Equation (\ref{normalizer}). Is clear that $X^{-1}=\sum_{i=1}^{n+1}x_i^{-1}E_{\sigma(i)}^T$ is the inverse matrix of $X$. Now, let $Y_i\in T_i$ be as follows
$$Y_i=
\left( \begin{array}{c|c|c}
                                  I_{i-1} &0_{i-1\times 2} &0_{i-1\times n-i}   \\\hline
                                  0_{2\times i-1}&e^{yh}&0_{2\times n-i}\\\hline
                                0_{n-i\times i-1} &0_{n-i\times 2}  & I_{n-i} 
                                  \end{array} \right).$$
After some simple calculations,  we have that $XY_iX^{-1}$ is a diagonal matrix and $\det(XY_iX^{-1})$ is equal to $1$. Thus, $X\in N(T)$.

On other hand, let's take an element of $N(T)$, $X=(x_{kl})$ with $X$ inverse being $X^{-1}=(y_{kl})$. Then, for all $Y\in T$, we have that $XYX^{-1}\in T$, including the generators of $T$. Let   $Y_i\in T_i$, for $1\leq i\leq n$, then
$$XY_iX^{-1}=X(z_{kl})=(\sum_{j=1}^{n+1}x_{kj}z_{jl}),$$
where
\[z_{kl}=\left\{\begin{array}{ll}y_{il}\,e^y & \mbox{if } k=i\\y_{i+1l}\,e^{-y} & \mbox{if } k=i+1\\y_{kl} & \mbox{otherwise.}\end{array}\right.\]
 Since $XY_iX^{-1}\in T$, then for $k\neq l$ we have
 $$0=\sum_{j=1}^{n+1}x_{kj}z_{jl}= \left(\sum_{\substack{1\leq j\leq n+1\\ j\not\in\{i,\,i+1\}}} x_{kj}y_{jl}\right)+x_{ki}y_{il}e^y+x_{ki+1}y_{i+1l}e^{-y}.$$
 But for $k\neq l$, we have that $\sum_{j=1}^{n+1}x_{kj}y_{jl}=0$. Thus, 
  \begin{eqnarray*}\sum_{j=1}^{n+1}x_{kj}z_{jl}&=&  x_{ki}y_{il}e^y+x_{ki+1}y_{i+1l}e^{-y}-(x_{ki}y_{il}+x_{ki+1}y_{i+1l})\\&=&x_{ki}y_{il}(e^y-1)+x_{ki+1}y_{i+1l}(e^{-y}-1).\end{eqnarray*}          
It follows that $x_{ki}y_{il}=0=x_{ki+1}y_{i+1l}$. Therefore, $x_{ki}y_{il}=0$ for all $i,\,l,\,k\in\{1,\ldots,n+1\}$ such that $k\neq l$. 
For any $1\leq i\leq n+1$, we can define the set $L_i\coloneqq \{x_{ij}| x_{ij}\neq0\}$, whose elements are in the $i$-th row. Since $\det(X)=1$ we have that $L_i\neq\emptyset$.  

Now, let $k_i=\min\{j|x_{ij}\in L_i\}$. Then, $x_{ik_{i}}y_{k_il}=0$ for all $l\neq i$. It follows that $y_{k_ii}\neq0$. Thus,  $x_{kk_{i}}y_{k_ii}=0$ for any $k\neq i$. Therefore $x_{ik_i}$ is the unique non zero element in the $k_i$-column of $X$. Thus, for $i\neq j$ we have that $k_i\neq k_j$, hence  $\{k_1,\ldots,k_{n+1}\}=\{1,\ldots,n+1\}$. There exist  $\sigma\in S_{n+1}$ such that $\sigma(i)=k_i$. Therefore 
$$X=\sum_{i=1}^{n+1}x_{i\Si(i)}E_{\sigma(i)}.$$
\end{proof}
For any $\Si\in S_{n+1}$, we use  $E_\Si$ to denote $\mbox{sgn}(\Si)E_{\sigma(1)}+\sum_{i=2}^{n+1}E_{\sigma(i)}$, where sgn$(\Si)$ is the sign of $\Si$ (which is  equal to $1$ if $\Si$ is even and is $-1$ if $\Si$ is odd). It is easy to see that $\det\left(E_\Si\right) =1$. Then, we have that the set $\{E_\Si\mid\Si\in S_{n+1}\}$ is a subset of  $N(T)$. 

On other hand, for any $\Si,\tau\in S_{n+1}$, it holds true that:
$$\left(\sum_{i=1}^{n+1}E_{\sigma(i)}\right)\left(\sum_{i=1}^{n+1}E_{\tau(i)}\right)=\sum_{i=1}^{n+1}E_{\Si\circ\tau(i)}.$$
Let $X\in N(T)$, and $\Si\in S_{n+1}$ be such that $X= \sum_{i=1}^{n+1}x_iE_{\sigma(i)}$. Then, $XE_\Si^{-1}$ is equal to $\diag(x_1\mbox{sgn}(\Si),\,x_2,\ldots,\,x_{n+1})$ and $\det(XE_\Si^{-1})=\det(X)\det(E_\Si^{-1})=1$. Thus, $XE_\Si^{-1}\in T$, and $[X]=[E_\Si]$, with $[X],\,[E_\Si]\in N(T)/T$.

We have that, for any  $\Si,\,\tau\in S_{n+1}$ such that $\Si\neq\tau$, $E_\Si E_\tau^{-1}$ is not a diagonal matrix, since $\Si\circ\tau^{-1}\neq\mbox{id}_{S_{n+1}}$. Therefore, for any $\Si,\,\tau\in S_{n+1}$, we have $[E_\Si]=[E_\tau]$ if and only if  $\Si=\tau$. Thus, we can conclude that
$$N(T)/T=\{[E_\Si]\mid\Si\in S_{n+1}\},$$
that is, $N(T)/T$ is isomophic to $S_{n+1}$, the simetric group. Given that at the end of section \ref{4} we saw that $S_{n+1}\cong \mathcal{W}$, then $\mathcal{W}\cong N(T)/T$.

A straightforward calculation reveals that the normalizer of $T_i$ in $G_i$, denoted by $N_i$, has the following form:                              
$$N_i=\left\langle\left\{\left( \begin{array}{c|c|c}
                                  I_{i-1} &0_{i-1\times 2} &0_{i-1\times n-i}   \\\hline
                                  0_{2\times i-1}&a*e-a^{-1}*f&0_{2\times n-i}\\\hline
                                0_{n-i\times i-1} &0_{n-i\times 2}  & I_{n-i} 
                                  \end{array} \right) | a\in\C,\, a\neq 0\right\}, T_i\right\rangle,$$   
from which it is evident that $N_i\subset N(T)$.
 
The orthogonal reflection corresponding to $A_i$ is denoted by $s_i\in \mathcal{W}$.
The proposition of Millson and Toledano \cite[Proposition 2.14]{Tol} will now be divided into three parts, each of which will be proven for our specific case.
\begin{prop} \label{1} For any choice of $\Si_i\in N_i\setminus T_i$, $1\leq i \leq n$, the assignment $S_i\mapsto\Si_i$ extends uniquely to a homomorphism $\Si:B_\g\lra N(T)$, that makes the diagram in Figure (\ref{diagramas}) commute.
\end{prop}  
\begin{proof} We must show that the $\Si_i$ satisfy the braid relations, which can be verified through simple calculations. Let $i$ be an element in $\{1,\cdots,\,n\}$, and
$$\Si_i\coloneqq \left( \begin{array}{c|c|c}
                                  I_{i-1} &0_{i-1\times 2} &0_{i-1\times n-i}   \\\hline
                                  0_{2\times i-1}&a_i*e-a_i^{-1}*f&0_{2\times n-i}\\\hline
                                0_{n-i\times i-1} &0_{n-i\times 2}  & I_{n-i} 
                                  \end{array} \right)\in N_i\setminus T_i.$$ 
Recall that
\[m_{ij}=\left\{\begin{array}{lcl}3 & \mbox{si}& |i-j|=1\\2 & \mbox{si}& |i-j|\geq 2.\end{array}\right.\]
We have the following cases.
\begin{description}
\item[ Case 1] When $|i-j|=1$, we have to prove that $\Si_i\Si_j\Si_i=\Si_j\Si_i\Si_j$. Without loss of generality, we can chose $j=i+1$. Due to the definition of matrices $\Si_i$ and $\Si_{i+1}$, it suffices to observe the behavior of their following submatrices in $\M_{3}(\C)$. 
$$M_i\coloneqq a_iE_{12}-a_i^{-1}E_{21}+E_{33}\mbox{,}\,\,\, M_{i+1}\coloneqq E_{11}+a_{i+1}E_{23}-a_{i+1}^{-1}E_{32}.$$
We have
\begin{eqnarray*}
M_iM_{i+1}M_i&=&(a_ia_{i+1}E_{13}-a_i^{-1}E_{21}+E_{33})M_i\\
&=&a_ia_{i+1}E_{13}-E_{22}+(a_ia_{i+1})^{-1}E_{31}\\
&=&(a_{i}E_{12}+a_{i+1}E_{23}+(a_ia_{i+1})^{-1}E_{31}) (E_{11}+a_{i+1}E_{23}-a_{i+1}^{-1}E_{32})\\
&=&(E_{11}+a_{i+1}E_{23}-a_{i+1}^{-1}E_{32})\left(a_iE_{12}-a_i^{-1}E_{21}+E_{33}\right) M_{i+1}\\
&=&M_{i+1}M_iM_{i+1}.
\end{eqnarray*}
It follows that $\Si_i\Si_j\Si_i=\Si_j\Si_i\Si_j$.
\item[Case 2]When $|i-j|\geq 2$, we have to prove that $\Si_i\Si_j=\Si_j\Si_i$. Without loss of generality, we can chose $i<j$. Consider the submatrices of $\Si_i$ and $\Si_j$ below:
\begin{eqnarray*}M'_i&=& \left( \begin{array}{c|c}
                                  I_{i-1} &0_{i-1\times 2}    \\\hline
                                  0_{2\times i-1}&a_i*e-a_i^{-1}*f\\
                                \end{array} \right)\in\M_{i+1}(\C),\\
                M_{ij}&=& \left( \begin{array}{c|c|c}
                                  I_{j-i-2} &0_{j-i-2\times 2} &0_{j-i-2\times n-j}   \\\hline
                                  0_{2\times j-i-2}&a_j*e-a_j^{-1}*f&0_{2\times n-j}\\\hline
                                0_{n-i\times j-i-2} &0_{n-j\times 2}  & I_{n-j} 
                                  \end{array} \right)\in\M_{n-i}(\C).
\end{eqnarray*}
Then

\begin{eqnarray*}
\Si_i\Si_j&=&\left( \begin{array}{c|c}
                                  M'_i &0_{i+1\times n-i}    \\\hline
                                  0_{n-i\times i+1}&I_{n-i}\\
                                \end{array} \right)
                                \left( \begin{array}{c|c}
                                  I_{i+1} &0_{i+1\times n-i}    \\\hline
                                  0_{n-i\times i+1}&M_{ij}\\
                                \end{array} \right)\\
&=&\left( \begin{array}{c|c}
                                  M'_i &0_{i+1\times n-i}    \\\hline
                                  0_{n-i\times i+1}&M_{ij}\\
                                \end{array} \right)\\
&=&\left( \begin{array}{c|c}
                                  I_{i+1} &0_{i+1\times n-i}    \\\hline
                                  0_{n-i\times i+1}&M_{ij}\\
                                \end{array} \right)\left( \begin{array}{c|c}
                                  M'_i &0_{i+1\times n-i}    \\\hline
                                  0_{n-i\times i+1}&I_{n-i}\\
                                \end{array} \right)\\
&=&\Si_j\Si_i.
\end{eqnarray*}
\end{description}
Therefore, the braid relations are satisfied.
\end{proof}

\begin{prop} If $\Si, \,\Si':B_\g\lra N(T)$ are the homomorphisms corresponding to the choices $\{\Si_i\}_{i=1}^n$ and $\{\Si'_i\}_{i=1}^n$, with $\Si_i,\Si'_i\in N_i\setminus T_i$, respectively, then there exists a $t\in T$ such that, for any $S\in B_\g$, $\Si(S)=t\Si'(S)t^{-1}$.
\end{prop}                         
\begin{proof} For any $i\in\{1,\cdots,n\}$, let $\Si_i$ and $\Si'_i$ be the following elements
\begin{eqnarray*}\Si_i&=&\left( \begin{array}{c|c|c}
                                  I_{i-1} &0_{i-1\times 2} &0_{i-1\times n-i}   \\\hline
                                  0_{2\times i-1}&a_i*e-a_i^{-1}*f&0_{2\times n-i}\\\hline
                                0_{n-i\times i-1} &0_{n-i\times 2}  & I_{n-i} 
                                  \end{array} \right),\\
\Si'_i&=& \left( \begin{array}{c|c|c}
                                  I_{i-1} &0_{i-1\times 2} &0_{i-1\times n-i}   \\\hline
                                  0_{2\times i-1}&b_i*e-b_i^{-1}*f&0_{2\times n-i}\\\hline
                                0_{n-i\times i-1} &0_{n-i\times 2}  & I_{n-i} 
                                  \end{array} \right).                                  
\end{eqnarray*}
We can chose $t_i\in T_i$, as follows
$$t_i= \left( \begin{array}{c|c|c}
                                  I_{i-1} &0_{i-1\times 2} &0_{i-1\times n-i}   \\\hline
                                  0_{2\times i-1}&\begin{array}{cc}
                                   a_i^{-1}b_i & 0 \\
                                   0 & a_ib_i^{-1}  
                                  \end{array}&0_{2\times n-i}\\\hline
                                0_{n-i\times i-1} &0_{n-i\times 2}  & I_{n-i} 
                                  \end{array} \right).$$
From our choice of $t_i$, it is clear that $\Si_i't_i=\Si_i$. Let $c_i$ be a complex number such that $t_i=e^{c_ih_i}$. Now, take $\{\La_i^\vee\}_{i=1}^n\subset \fh$  the co-weights defined by $\A_i(\La_j^\vee)=\D_{ij}$, where $\D_{ij}$ is Kronecker's delta. In our case, $\La_j^\vee=\diag(h_{1j},\,h_{2j},\cdots,\,h_{n+1\,j})$ with $$h_{ij}= \left\{ \begin{array}{lcl}
\frac{n+1-j}{n+1}&\mbox{ si}&i\leq j\\  
& &\\ \frac{-j}{n+1}&\mbox{ si}&i>j.
\end{array}
\right.$$
 We choose a diagonal matrix $t\in T$ as follow $t=e^{-\sum_jc_j\La_j^\vee}$. Since $t$ and its inverse matrix are diagonal matrices, we only need to consider the following product of submatrices: 
\begin{eqnarray*}
&&\left( \begin{array}{cc}e^{-\sum_jc_jh_{ij}}&0\\0&e^{-\sum_jc_jh_{i+1,\, j}}\\\end{array} \right)\left( \begin{array}{cc}0&b_i\\-b_i^{-1}&0\\\end{array} \right)\left( \begin{array}{cc}e^{\sum_jc_jh_{ij}}&0\\0&e^{\sum_jc_jh_{i+1,\, j}}\\\end{array} \right)\\
&=&\left( \begin{array}{cc}0& b_i \exp({\sum_jc_j(h_{i+1,\, j}-h_{i j})}\\b_i^{-1} \exp(\sum_jc_j(h_{i j}-h_{i+1,\, j}))&0\\\end{array} \right)\\
&=&\left( \begin{array}{cc} 0 & b_i \exp({\sum_jc_j(- \A_i(\La_ j^\vee))})\\ -b_i^{-1} \exp({\sum_jc_j(\A_i(\La_ j^\vee))})&0\end{array} \right)\\
&=&\left( \begin{array}{cc}0&b_ie^{-c_i}\\-b_i^{-1}e^{c_i}&0\\\end{array} \right)\\
&=&\left( \begin{array}{cc}0&b_i\\-b_i^{-1}&0\\\end{array} \right)\left( \begin{array}{cc}e^{c_i}&0\\0&e^{-c_i}\\\end{array} \right).\\
\end{eqnarray*}
 Then, we can conclude that
\begin{eqnarray*}
 t \Si'_it^{-1}= e^{-\sum_jc_j\La_j^\vee}\,\,\Si'_i\,\, e^{\sum_jc_j\La_j^\vee}=\Si'_i\,\,e^{c_ih_{\A_i}} =\Si'_i\,\,t_i = \Si_i.
\end{eqnarray*}                                             
Thus, we have for all $i\in\{1,\cdots,\,n\}$ that $t\, \Si'_i\, t^{-1}=\Si_i$, with $t=e^{-\sum_jc_j\La_j^\vee}$. 

Therefore, $t\, \Si'\, t^{-1}=\Si$.
\end{proof}

\begin{prop} \label{propi2} For any $\Si:B_\g\lra N(T)$, the homomorphism corresponding to the choice $\{\Si_i\}_{i=1}^n$, is subjected to the relations:
\begin{eqnarray}\label{2.9}\underbrace{\Si_i\Si_j\cdots}_{m_{ij}}&=&\underbrace{\Si_j\Si_i\cdots}_{m_{ij}}\\
\label{2.10}\Si_i^2\Si_j^2&=&\Si_j^2\Si_i^2\\
\label{2.11}\Si_i^4&=&1\\
\label{2.12}\Si_i\Si_j^2\Si_i^{-1}&=&\Si_j^2\Si_i^{-2\langle \A_j,h_i\rangle} \end{eqnarray}
for any $1\leq i\neq j\leq n$, where the number $m_{ij}$ of factors on each side of (\ref{2.9}) is equal to the order of $s_is_j\in \mathcal{W}$.
\end{prop}
\begin{proof} 
In proposition \ref{1}, we have already proved that $\Si_i$ satisfies equality (\ref{2.9}). Additionally, we note that for all $1\leq j \leq n$, if we take 
$$z_i={\scriptstyle\left( \begin{array}{c|c|c}
                                  I_{i-1} &0_{i-1\times 2} &0_{i-1\times n-i}   \\\hline
                                  0_{2\times i-1}&a*e-a^{-1}*f&0_{2\times n-i}\\\hline
                                0_{n-i\times i-1} &0_{n-i\times 2}  & I_{n-i} 
                                  \end{array} \right),}$$
which can also be viewed as $z_i=\left(\sum_{\substack{1\leq l\leq n+1\\ l\not\in\{i,\,i+1\}}} E_{l,\,l}\right)+(aE_{i,\,i+1}-a^{-1}E_{i+1,\,i})$, it follows that 
\begin{eqnarray*}
z_i^{-1}&=&\left(\sum_{\substack{1\leq l\leq n+1\\ l\not\in\{i,\,i+1\}}} E_{l,\,l}\right)+(-aE_{i,\,i+1}+a^{-1}E_{i+1,\,i}),\\
z_i^2&=&\left(\sum_{\substack{1\leq l\leq n+1\\ l\not\in\{i,\,i+1\}}} E_{l,\,l}\right)-(E_{i,\,i}+E_{i+1,\,i+1}).
\end{eqnarray*}
It is then evident that the $\Si_i$ satisfy the relations (\ref{2.10}) y (\ref{2.11}).                                 
For the last equation, we can take $1\leq i\neq j\leq n$, and thus have two cases:
\begin{description}
\item[ Case 1] If $|j-i|=1$, then $\langle \A_j,h_i\rangle=-1$. Let's suppose  $i=j-1$, this implies that:  
\begin{eqnarray*}
&&\Si_i\Si_j^2\Si_i^{-1}\\
&=& \Si_{j-1}\Si_j^2\Si_{j-1}^{-1}\\
&=& {\left(\sum_{\substack{1\leq l\leq n+1\\ l\not\in\{j-1,\,j\}}} E_{l,\,l}+a_{j-1}E_{j-1j}-a_{j-1}^{-1}E_{jj-1}\right)  \left(\sum_{\substack{1\leq l\leq n+1\\ l\not\in\{j,\,j+1\}}} E_{l,\,l}-E_{jj}-E_{j+1\,j+1}\right) }\Si_{j-1}^{-1}\\
&=&\left(\sum_{\substack{1\leq l\leq n+1\\ l\not\in\{j-1,\,j,\,j+1\}}} E_{l,\,l}-a_{j-1}E_{j-1j}-a_{j-1}^{-1}E_{jj-1}-E_{j+1j+1}\right)\Si_{j-1}^{-1}\\
&=& \left(\sum_{\substack{1\leq l\leq n+1\\ l\not\in\{j-1,\,j+1\}}} E_{l,\,l}\right)-(E_{j-1j-1}+E_{j+1\,j+1})\\
&=& \left(\sum_{\substack{1\leq l\leq n+1\\ l\not\in\{j,\,j+1\}}} E_{l,\,l}-E_{jj}-E_{j+1\,j+1}\right)\left(\sum_{\substack{1\leq l\leq n+1\\ l\not\in\{j-1,\,j\}}} E_{l,\,l}-E_{j-1j-1}-E_{j\,j}\right)\\
&=&\Si_j^2\Si_{j-1}^2=\Si_j^2\Si_{i}^2=\Si_j^2\Si_{i}^{(-2)(-1)}.
\end{eqnarray*}  
The case in which $i=j+1$ is analogous to the previous one.
\item[Case 2]If $|j-i|\geq 2$, then $\langle \A_j,h_i\rangle=0$. 

Thus, $\Si_j^2\Si_i^{-2\langle \A_j,h_i\rangle}=\Si_j^2\Si_i^{(-2)(0)}=\Si_j^2\Si_i^{(0)}=\Si_j^2$. Since $|j-i|\geq 2$, we have that $\{i,i+1\}\cap\{j,j+1\}=\emptyset$, from where we have that
\begin{eqnarray*}
&&\Si_i\Si_j^2\Si_i^{-1}\\
&=& {\left(\sum_{\substack{1\leq l\leq n+1\\ l\not\in\{i,\,i+1\}}} E_{l,\,l}+a_{i}E_{i,\,i+1}-a_{i}^{-1}E_{i+1i}\right)  \left(\sum_{\substack{1\leq l\leq n+1\\ l\not\in\{j,\,j+1\}}} E_{l,\,l}-E_{jj}-E_{j+1\,j+1}\right) }\Si_{i}^{-1}\\
&=&{\left(\sum_{\substack{1\leq l\leq n+1\\ l\not\in\{i,\,i+1,\,j,\,j+1\}}} E_{l,\,l}+a_{i}E_{i,\,i+1}-a_{i}^{-1}E_{i+1i}-E_{jj}-E_{j+1\,j+1}\right) }\Si_{i}^{-1}\\
&=& \left(\sum_{\substack{1\leq l\leq n+1\\ l\not\in\{j,\,j+1\}}} E_{l,\,l}-E_{jj}-E_{j+1\,j+1}\right)\\&=&\Si_j^2.
\end{eqnarray*} 
\end{description}
Therefore, it can be concluded that equality (\ref{2.12}) has been achieved.
\end{proof}

\begin{cor}
    Let $n$ be a positive integer and $\tau_i$'s be the automorphisms  of $\fraksl(n+1,\C)$ defined in  (\ref{tau}). Then the automorphisms $\tau_i$, $1\leq i\leq n$ are subjected to the relations (\ref{2.9})-(\ref{2.12}).
\end{cor}
\begin{proof}
As we saw in Remark (\ref{remark1}), for any $X,\,Y$ matrices in $\M_{n+1}(\C)$, we have that $\exp(\ad_X)(Y)$ is equal to $\exp(X)Y\exp(-X)$, including the  particular case where $X=e_i$ or $f_i$. 

For all $1\leq i\leq n$, let $\Si_i$ be equal to $\exp(e_i)\exp(-f_i)\exp(e_i)$.  Then,
\begin{eqnarray*}
\tau_i(Y)&=&\exp(\ad_{e_i})\exp(\ad_{-f_i})\exp(\ad_{e_i})(Y)\\
&=&\exp(e_i)[\exp(-f_i)[\exp(e_i)Y\exp(-e_i)]\exp(f_i)]\exp(-e_i)\\
&=&(\exp(e_i)\exp(-f_i)\exp(e_i))Y(\exp(-e_i)\exp(f_i)\exp(-e_i))\\
&=&\Si_iY\Si_i^{-1}.
\end{eqnarray*}
From the definition of element $e_i$ and $f_i$, we have that one of the possible choices of elements of $N_i\setminus T_i$ is
\begin{eqnarray*}
\Si_i&=&(I_{n+1}+E_{i,\, i+1})(I_{n+1}-E_{i+1,\, i})(I_{n+1}+E_{i,\, i+1})\\
&=&I_{n+1}+E_{i,\, i+1}-E_{i+1, i}-E_{i,\, i}-E_{i+1,\, i+1}\\
&=& \left( \begin{array}{cccc}
                                  I_{i-1} & &  &    \\
                                  & 0 & 1 &   \\
                                  & -1 & 0 &  \\
                                &  &  & I_{n-i} 
                                  \end{array} \right).
\end{eqnarray*}
Thus, $\{\Si_i\}_{i=1}^n$ satisfies the relations  (\ref{2.9})-(\ref{2.12}). Since $\tau_i(Y)=\Si_iY\Si_i^{-1}$, we can conclude that the standard automorphisms satisfy the desired relations. 
\end{proof}

\section*{Acknowledgment}
I would like to thank CONACYT for the Support for the Training of Human Resources with registration number 25412 of the Research Project CB-2014/239255 and Christof Geiss for his support, guidance and encouragement. Part of this work was done during my Master's studies at the Universidad Nacional Autónoma de México.

\end{document}